\documentclass[11pt]{amsart}
\usepackage{amssymb,latexsym,amsmath,graphicx,graphics,epic,eepic}

\theoremstyle{plain}
\newtheorem{theorem}{Theorem}
\numberwithin{theorem}{section}
\newtheorem{lemma}[theorem]{Lemma}

\theoremstyle{definition}

\newtheorem{question}[theorem]{Question}
\newtheorem{remark}[theorem]{Remark}

\newtheorem{acknowledgments}{Acknowledgments\ignorespaces}

\newcommand{\C}{{\mathbb C}}

\newcommand{\Z}{{\mathbb Z}}

\renewcommand{\P}{{\mathbb P}}
\newcommand{\s}{{\mathbb S}}

\begin{document}
\title[Self-intersection of rational unicuspidal curves]
{A note on the self-intersection of rational unicuspidal curve with one Puiseux pair}
\author{Weimin Chen, Yusupjan Ouyang and Sam Silver}
\subjclass[2000]{}
\keywords{}
\date{\today}
\maketitle
\begin{abstract}
For any $(p,q)$, we proved the existence of a rational unicuspidal curve with one Puiseux pair $(p,q)$ in an algebraic surface, with a self-intersection which realizes the upper bound $m_{p,q}$ conjectured by the first-named author in \cite{C}, thus strengthening the symplectic version of the result in \cite{C}.
These  ``optimal" curves were obtained by examining two infinite families of rational bicuspidal curves, one in $\C\P^2$ and one in $\C\P^1\times \C\P^1$. In order to facilitate the computations, we derived certain recursive identities, and as a byproduct, we obtained a new formula for the bound $m_{p,q}$, which is more amenable to computations and gives us better insight concerning the nature of the bound. As a byproduct of this investigation, a formula for the last entry of the multiplicity sequence of the singularity was also found. 
\end{abstract}

\section{Introduction}

In \cite{C}, motivated by a question in symplectic geometry, the first-named author proposed an upper bound for the self-intersection of a rational unicuspidal curve with one Puiseux pair in any algebraic surface, which depends only on the Puiseux pair of the singularity. More concretely, for any 
$1<p<q$ where $p,q$ are co-prime, we set 
$$
m_{p,q}:=pq +\max (\lfloor \frac{p}{p-p^\prime}\rfloor, \lfloor \frac{q}{q-q^\prime}\rfloor), 
$$
where the integers $p^\prime,q^\prime$ are uniquely determined by the conditions 
$$
1\leq p^\prime<p, 1\leq q^\prime<q, \mbox{ and } pq^\prime+qp^\prime=pq+1.
$$
Then it is conjectured that for a rational unicuspidal curve $C$ with one Puiseux pair $(p,q)$ in any algebraic surface $X$, one must have $C\cdot C\leq m_{p,q}$ (cf. \cite{C}, Conjecture 1.15). 

In this note, we first present a formula for the bound $m_{p,q}$, which is more friendly for computation than the definition of $m_{p,q}$ and gives better insight for $m_{p,q}$ in connection with the structure of the cuspidal singularity. 

\begin{theorem}
Let $\ell(p,q)>1$ be the last entry of the multiplicity sequence of a cuspidal singularity with one Puiseux pair $(p,q)$. Then
\begin{itemize}
\item [{(1)}] $m_{p,q}=\left\{ \begin{array}{ll}
pq+\ell(p,q)+1 & \mbox{ if $q=p+1$ or $p>2$ and $q=-1 \pmod{p}$}\\
pq+\ell(p,q)  & \mbox{ in all other cases. } 
\end{array}
\right.$
\item [{(2)}] $\ell(p,q)=\lfloor \frac{q-1}{\min(u,q-u)} \rfloor$, where $up=1\pmod{q}$, $1<u<q$.
\end{itemize}
\end{theorem}

\begin{remark}
(1) Denote by $M_{p,q}$ the sum of the squares of the entries in the multiplicity sequence associated to the minimal resolution of a cuspidal singularity with one Puiseux pair $(p,q)$. Then $\ell(p,q)$ and
$M_{p,q}$ are related by the following equation
$$
\ell(p,q)=pq-M_{p,q}. 
$$
(See Lemma 2.1(2).) Consequently, one can also express $m_{p,q}$ in terms of $M_{p,q}$, i.e., 
$$
\mbox{ $m_{p,q}=M_{p,q}+2\ell(p,q)+1$ or $m_{p,q}=M_{p,q}+2\ell(p,q)$. }
$$
These expressions provide close connections between the bound $m_{p,q}$ and the topological structure of the cuspidal singularity, i.e., its multiplicity sequence, which we hope are more useful for a
proof of Conjecture 1.15 in \cite{C}. On the other hand, the formula for the last entry $\ell(p,q)$, i.e.,
$$
\ell(p,q)=\lfloor \frac{q-1}{\min(u,q-u)} \rfloor, \mbox{ where $up=1\pmod{q}$, $1<u<q$,}
$$
which seems to be unknown before, gives a simpler formula for directly computing the bound 
$m_{p,q}$ than the definition $m_{p,q}:=pq +\max (\lfloor \frac{p}{p-p^\prime}\rfloor, \lfloor \frac{q}{q-q^\prime}\rfloor)$ (see also Lemma 2.4).

(2) Golla and Starkston introduced an upper bound for the self-intersection of a symplectic singular rational curve in a symplectic $4$-manifold in terms of its singularities (cf. \cite{GS}, Proposition 6.2). Specializing to the case of unicuspidal curves with one Puiseux pair $(p,q)$, the bound of Golla and Starkston is given by $M_{p,q}+2\ell(p,q)+1$. It follows easily that the bound $m_{p,q}$ is strictly 
smaller than the bound of Golla and Starkston, except for the cases where $(p,q)=(m-1,m)$ or 
$(p,q)=(m,km-1)$ for some $m\geq 3$ and $k\geq 2$, in which the two bounds coincide.  This was actually proved in \cite{C}, Lemma 6.11. We note that Theorem 1.1 says more: when $m_{p,q}$ is strictly smaller than the bound of Golla and Starkston, the two bounds differ by exactly $1$!

(3) If we let $(m_1,m_2,\cdots,m_k)$, where each $m_i>1$, be the multiplicity sequence of the 
cuspidal singularity with one Puiseux pair $(p,q)$, then the formula for the $\delta$-invariant gives the following identity
$$
(p-1)(q-1)=\sum_{i=1}^k m_i (m_i-1). 
$$
Note that $M_{p,q}=\sum_{i=1}^k m_i^2=pq-\ell(p,q)$, which implies the following formulas for 
$M_{p,q}$ and the sum of the multiplicities $m_i$, computable directly from $p$ and $q$: 
$$
M_{p,q}=pq-\lfloor \frac{q-1}{\min(u,q-u)} \rfloor,  \;\; \;\;
\sum_{i=1}^k m_i=p+q-1-\lfloor \frac{q-1}{\min(u,q-u)} \rfloor,
$$
where $up=1\pmod{q}$, $1<u<q$.
\end{remark}

For the second result in this note, we showed that for any $(p,q)$, there is a rational unicuspidal curve
with one Puiseux pair $(p,q)$ whose self-intersection realizes the upper bound $m_{p,q}$. This is a
strengthening of Theorem 1.17(3) in \cite{C}, where the existence of such an ``optimal" symplectic rational unicuspidal curve is established. 

\begin{theorem}
For any $(p,q)$, there is a rational surface which contains a rational unicuspidal curve $C$ with one Puiseux pair $(p,q)$ such that $C\cdot C=m_{p,q}$.
\end{theorem}

\begin{remark}
Theorem 1.3 has been verified for some special cases in \cite{C}. More concretely, for the case where $q=p+1$ or $4p-1$, the curve $C$ in Theorem 1.3 was shown to be realized by the corresponding rational unicuspidal curves in $\C\P^2$, a complete classification of which is given in \cite{F}. It was also verified for $p=2$; in particular, for the case of $p=2$ and $q>3$, the curve $C$ was constructed from a family of rational tricuspidal curves in $\C\P^2$ due to Flenner and Zaidenberg \cite{FZ}. See Examples 6.7 and  6.8 in \cite{C}. In addition, Theorem 1.3 was also verified for a few other low values of $p$, i.e., for $p=3,4,5$, by examining two specific families of rational bicuspidal curves, one in 
$\C\P^2$ and one in $\C\P^1\times \C\P^1$. For more details, see Example 6.9 of \cite{C}, where the case of $p=3$ was explicitly worked out. With the help of certain recursive relations behind the formula of the bound $m_{p,q}$ in Theorem 1.1, we are able to extend the calculations on the two infinite families of rational bicuspidal curves considered in Example 6.9 of \cite{C} for any values of $p$ and $q$, thus verifying Theorem 1.3 for the remaining cases. 

\end{remark}

For notational simplicity, we shall introduce the following ad hoc terminology: letting $C$ be a rational unicuspidal curve with one Puiseux pair $(p,q)$, we call the difference $m_{p,q}-C\cdot C$ the 
{\bf gap} for the curve $C$.

We now state the last result in this note, which is concerned with the following two families of rational bicuspidal curves considered in Example 6.9 of \cite{C}: 

\vspace{2mm}

{\bf (i):} The family of bicuspidal curves in $\C\P^2$: Let $1<p<q$ such that $gcd(p,q)=1$ and $q-p>1$.
We consider the rational curves in $\C\P^2$ defined by the following parametrizations:
$$
[z^p: z^q: 1], \mbox{ where $z\in \C\sqcup\{\infty\}$}. 
$$
Each curve has  two cuspidal singularities: one at $z=0$, with one Puiseux pair $(p,q)$, and the other
at $z=\infty$, with one Puiseux pair $(q-p,q)$. 

\vspace{2mm}

{\bf (ii):} The family of bicuspidal curves in $\C\P^1\times \C\P^1$: Let $1<p<q$ such that 
$gcd(p,q)=1$, and let $\alpha\in \Z_{>0}$ such that $q<\alpha p$. We consider the 
rational curves in $\C\P^1\times \C\P^1$ defined by the following parametrizations (cf. \cite{Moe}):
$$
([z^p:1], [z^q+z^{\alpha p} :1]), \mbox{ where $z\in \C\sqcup\{\infty\}$}.
$$
Each curve has two cuspidal singularities: one at $z=0$, with one Puiseux pair $(p,q)$, and the other
at $z=\infty$, with one Puiseux pair $(p,2\alpha p-q)$. 

\vspace{2mm}

With the preceding understood, we completely classified in the following theorem the gap patterns of the rational unicuspidal curves produced from the preceding two families of bicuspidal curves
by resolving one of the two singularities of the bicuspidal curves. 

\begin{theorem}
(1) For the family of bicuspidal curves in $\C\P^2$, the following are the only possibilities, where we note that case (i) and case (ii) are mutually exclusive. 
\begin{itemize}
\item [{(i)}] If $p>2$ and $p| (q+1)$, then the case resolving the singularity at $z=0$ has gap $0$ and the other case (i.e., resolving the singularity at $z=\infty$) has gap $1$. 
\item [{(ii)}] If $q-p>2$ and $(q-p) | (q+1)$, then the case resolving the singularity at $z=0$ has gap $1$ and the other case has gap $0$. 
\item [{(iii)}] In all other cases, the gap is $0$ in either case. 
\end{itemize}

(2) For the family of bicuspidal curves in $\C\P^1\times C\P^1$, the following are the only possibilities
for $p>2$, where we note that the cases (i), (ii) and (iii) are mutually exclusive. 
\begin{itemize}
\item [{(i)}] If $q-p=1$, then the case resolving the singularity at $z=0$ has gap $0$ and the other case (i.e., resolving the singularity at $z=\infty$) has gap $2$. 
\item [{(ii)}] If $q-p>1$ and $p | (q-1)$, then the case resolving the singularity at $z=0$ has gap $0$ and the other case has gap $1$. 
\item [{(iii)}] If $p| (q+1)$, then the case resolving the singularity at $z=0$ has gap $1$ and the other case has gap $0$. 
\item [{(iv)}] In all other cases, the gap is $0$ in either case. 
\end{itemize}
For $p=2$, if $q=3$, then the case resolving the singularity at $z=0$ has gap $0$ and the other case has gap $1$, and if $q>3$, the gap is $0$ in either case. 
\end{theorem}

We remark that the initial examination of these bicuspidal curves in Example 6.9 of \cite{C} suggested that the gaps were either $0$ or $1$. More extensive calculations performed by the second-named author using a Python program revealed a more refined pattern in the cases tested: at least one of the two gaps was always $0$, while the other was $0$ or $1$ except in the gap-$2$ case described in Theorem 1.5(2)(i). These observations suggested the classification in Theorem 1.5 and helped guide its subsequent proof.

Note that since in all the possible cases in Theorem 1.5, the gap is always non-negative, it follows easily that the two families of bicuspidal curves do not produce any counterexamples of 
Conjecture 1.15 in \cite{C}. On the other hand, Theorem 1.3 follows as a corollary of Theorem 1.5. 

\vspace{2mm}

{\bf Proof of Theorem 1.3:} 

\vspace{2mm}

Since the case where $q=p+1$ or $p=2$ has already been verified in \cite{C}, for simplicity we may assume $p>2$ and $q-p>1$ without loss of generality. With this understood, we consider separately the following two different scenarios. 

First, assume that $q+1$ is not divisible by $p$. In this case, we let $C$ be the rational unicuspidal curve produced from the $\C\P^2$-family by resolving the singularity at $z=\infty$. Then $C$ has 
Puiseux pair $(p,q)$. Note that we are either in case (ii) or case (iii) of Theorem 1.5(1). Since the gap in both case (ii) and case (iii) is $0$, we have $C\cdot C=m_{p,q}$ for the curve $C$. 

Next, assume that $q+1$ is divisible by $p$. In this case, if we let $C$ be the rational unicuspidal curve produced from the $\C\P^1\times \C\P^1$-family by resolving the singularity at $z=\infty$, then $C$ has Puiseux pair $(p,q)$. Note that since $p>2$, we are in case (iii) of Theorem 1.5(2). As the gap is $0$, we have $C\cdot C=m_{p,q}$ for the curve $C$. 

The proof of Theorem 1.3 is completed. 

\vspace{2mm}

Our work naturally raised the following question.

\begin{question}
Fixing any pair $(p,q)$, what is the complete list of $(X,C)$, where $C$ is a rational unicuspidal curve 
with one Puiseux pair $(p,q)$ and $X$ is an algebraic surface containing $C$, such that $C$ attains the maximum in self-intersections among all such curves with one Puiseux pair $(p,q)$? Here we assume 
$(X,C)$ is minimal in the following sense: there are no exceptional $(-1)$-curves in the complement 
$X\setminus C$. Note that since $C$ attains the maximum in self-intersections, there are also no exceptional $(-1)$-curves which intersect $C$ transversely (hence away from the singularity) at one point. 
\end{question}

\begin{remark}
(1) By Hartshorne \cite{Hartsh}, a singular rational curve cannot have arbitrarily large self-intersection. Hence the pair $(X,C)$ where $C$ has the maximal self-intersection is well-defined. With this understood, note that if Conjecture 1.15 of \cite{C} is true, then Theorem 1.3 implies that the maximal self-intersection for the rational unicuspidal curves with one Puiseux pair $(p,q)$ is given by $m_{p,q}$. 
From the work of Hartshorne \cite{Hartsh}, it is a natural question to classify the singular rational curves
which have the maximal self-intersection fixing a topological type of the singularities, and Question 1.6 is the special case concerning curves with the simplest possible singularity types, in which case there is also a conjectured answer for the maximal self-intersection. 

(2) Alternatively, one may classify the singular rational curves $S$ in $\C\P^2$ or in a Hirzebruch surface, where $S$ is such that if we resolve all but one of the singular points of $S$ (cuspidal or non-cuspidal), we obtain a pair $(X,C)$ in Question 1.6. We note that the proof (or verification) of Theorem 1.3, either in this note or in \cite{C}, involved some natural examples of such singular curves $S$, namely,  the two families of bicuspidal curves which are fully studied in this note, or the family of tricuspidal curves of Flenner and Zaidenberg \cite{FZ} and the unicuspidal curves in $\C\P^2$ from \cite{F}, which were used in \cite{C}. It would be interesting to find more examples of such families of rational curves, as it is likely that the set of the singular curves $S$ is given by a small number of infinite families of rational cuspidal curves, which may be completely classified. 

(3) Question 1.6 is an ``algebraic geometry" manifestation of a corresponding question in symplectic geometry, i.e., classifying the symplectic fillings of the relevant $\s^1$-invariant contact structure (or equivalently, the corresponding Li-Mak contact structure \cite{LM}). See \cite{C} for more details. 
\end{remark}

\begin{acknowledgments}
This research was carried out as part of the Summer 2026 REU program in the Department of Mathematics and Statistics at the University of Massachusetts Amherst. The second and third named authors thank the program for funding their participation and for providing the opportunity to work on this project.
\end{acknowledgments}

\section{Recursive relations and formulas concerning the bound $m_{p,q}$}

Let $1<p<q$ where $gcd(p,q)=1$. We define $\tilde{p}=p$ and $\tilde{q}=q-p$, and assume 
$\tilde{q}>1$. Note that if $C$ is a germ of cuspidal singularity with one Puiseux pair $(p,q)$, and let
$\tilde{C}$ denote the proper transform of $C$ after blowing up at the singularity, then 
$\tilde{C}$ is a germ of cuspidal singularity with one Puiseux pair $(\tilde{p},\tilde{q})$ or
$(\tilde{q},\tilde{p})$. 

With the preceding understood, we introduce the following notations
$$
R_{p,q}:=pq-M_{p,q} \mbox{ and } B_{p,q}:=m_{p,q}-pq=\max (\lfloor \frac{p}{p-p^\prime}\rfloor, \lfloor \frac{q}{q-q^\prime}\rfloor), 
$$
where $p^\prime,q^\prime$ are uniquely determined by the conditions $1\leq p^\prime<p$,
$1\leq q^\prime<q$, and $pq^\prime+qp^\prime=pq+1$. For notational simplicity, we shall not require 
the pair $(a,b)$ in any of $M_{a,b}$, $R_{a,b}$ or $B_{a,b}$ obey the condition $a<b$. With this 
understood, for convenience we state in the following lemma some simple observations concerning $R_{p,q}$, which may be well-known among the experts. 

\begin{lemma}
Assume $1<p<q$ where $gcd(p,q)=1$ and set $\tilde{p}=p$ and $\tilde{q}=q-p$.
Then $R_{p,q}=R_{\tilde{p},\tilde{q}}$ if $\tilde{q}>1$. As a consequence, one has 
\begin{itemize}
\item [{(1)}] $R_{p,q}=R_{q-p,q}$ when $q-p>1$. 
\item [{(2)}] Let $\ell(p,q)>1$ be the last entry of the multiplicity sequence of a cuspidal singularity with one Puiseux pair $(p,q)$. Then $\ell(p,q)=R_{p,q}=pq-M_{p,q}$. As a corollary, note that 
$\ell(p,q)=\ell(q-p,q)$ when $q-p>1$. 
\end{itemize}
\end{lemma}

\begin{proof}
Assume $\tilde{q}>1$. Then we note that $\tilde{p}\tilde{q}=pq-p^2$, $M_{\tilde{p},\tilde{q}}=M_{p,q}-p^2$, from which it follows immediately that $R_{p,q}=R_{\tilde{p},\tilde{q}}$. 

To see $R_{p,q}=R_{q-p,q}$ when $q-p>1$, we note that $R_{p,q}=R_{p,q-p}$ and 
$R_{q-p,q}=R_{q-p, q-(q-p)}=R_{q-p,p}$, which implies that $R_{p,q}=R_{q-p,q}$.

Finally, for $\ell(p,q)=R_{p,q}=pq-M_{p,q}$, note that since $p,q$ are co-prime, under finitely many recursive reductions of the form $(a,b)$ to $(\tilde{a},\tilde{b})$, where $a<b$, $\tilde{a}=a$, 
and $\tilde{b}=b-a$, one reaches to a pair of the form $(\ell,\ell+1)$,
where $\ell=\ell(p,q)>1$. It follows easily that $R_{p,q}=R_{\ell,\ell+1}=\ell$, which implies 
$\ell(p,q)=R_{p,q}=pq-M_{p,q}$. The identity $\ell(p,q)=\ell(q-p,q)$ is simply $R_{p,q}=R_{q-p,q}$.

\end{proof}

On the other hand, the corresponding recursive relation for $B_{p,q}$ is slightly more complicated. 
First, the following lemma is a strengthening of Lemma 6.10 in \cite{C}.

\begin{lemma}
Assume $1<p<q$ where $gcd(p,q)=1$ and set $\tilde{p}=p$ and $\tilde{q}=q-p$, with $\tilde{q}>1$.
Then one has $0\leq B_{\tilde{p},\tilde{q}}-B_{p,q}\leq 1$. Moreover, $B_{p,q}=B_{\tilde{p},\tilde{q}}-1$
if and only if either $\tilde{q}=\tilde{p}+1$ or $\tilde{p}>\tilde{q}>2$ and $\tilde{p}=-1 \pmod{\tilde{q}}$. 
\end{lemma}

\begin{proof}
Let $\tilde{p}^\prime,\tilde{q}^\prime$ be the integers uniquely determined by the conditions 
$1\leq \tilde{p}^\prime<\tilde{p}$, $1\leq \tilde{q}^\prime<\tilde{q}$, and 
$\tilde{p}\tilde{q}^\prime+\tilde{q}\tilde{p}^\prime=\tilde{p}\tilde{q}+1$. Then as one observed in \cite{C},
Lemma 6.10, it is easy to check that 
$$
\tilde{p}^\prime =p^\prime, \;\; \tilde{q}^\prime=q^\prime-p+p^\prime,
$$
which imply easily that
$$
\frac{\tilde{p}}{\tilde{p}-\tilde{p}^\prime}=\frac{p}{p-p^\prime}, \;\; 
\frac{\tilde{q}}{\tilde{q}-\tilde{q}^\prime}-\frac{q}{q-q^\prime}=\frac{1}{(\tilde{q}-\tilde{q}^\prime)(q-q^\prime)}. 
$$
With $B_{p,q}=\max (\lfloor \frac{p}{p-p^\prime}\rfloor, \lfloor \frac{q}{q-q^\prime}\rfloor)$ and
$B_{\tilde{p},\tilde{q}}=\max (\lfloor \frac{\tilde{p}}{\tilde{p}-\tilde{p}^\prime}\rfloor, \lfloor 
\frac{\tilde{q}}{\tilde{q}-\tilde{q}^\prime}\rfloor)$, it follows immediately that
$0\leq B_{\tilde{p},\tilde{q}}-B_{p,q}\leq 1$.

Next, assume $B_{p,q}=B_{\tilde{p},\tilde{q}}-1$. Since $B_{\tilde{p},\tilde{q}}>B_{p,q}$, we must have 
$$
\frac{\tilde{q}}{\tilde{q}-\tilde{q}^\prime}> 
\frac{\tilde{p}}{\tilde{p}-\tilde{p}^\prime}=\frac{{p}}{{p}-{p}^\prime} \mbox{ and }
B_{\tilde{p},\tilde{q}}=\lfloor \frac{\tilde{q}}{\tilde{q}-\tilde{q}^\prime}\rfloor>\frac{q}{q-q^\prime}.
$$
It follows easily that
$$
\frac{1}{(\tilde{q}-\tilde{q}^\prime)(q-q^\prime)}\geq 
\lfloor \frac{\tilde{q}}{\tilde{q}-\tilde{q}^\prime}\rfloor-\frac{q}{q-q^\prime}\geq \frac{1}{q-q^\prime}, 
$$
which implies that $\tilde{q}-\tilde{q}^\prime=1$ must be true. Substitute $\tilde{q}^\prime=\tilde{q}-1$ 
in $\tilde{p}\tilde{q}^\prime+\tilde{q}\tilde{p}^\prime=\tilde{p}\tilde{q}+1$, we arrive at
$\tilde{q}\tilde{p}^\prime=\tilde{p}+1$. 

If $\tilde{q}>\tilde{p}$, then one must have $\tilde{q}=\tilde{p}+1$. On the other hand, if $\tilde{q}<\tilde{p}$, then we observe that  $\tilde{p}=-1 \pmod{\tilde{q}}$. It remains to show that 
$\tilde{q}>2$. To see this, note that 
$$
\frac{\tilde{q}}{\tilde{q}-1}<\frac{\tilde{q}\tilde{p}^\prime-1}{(\tilde{q}-1)\tilde{p}^\prime-1}=
\frac{\tilde{p}}{\tilde{p}-\tilde{p}^\prime}<\frac{\tilde{q}}{\tilde{q}-\tilde{q}^\prime}=\tilde{q},
$$
from which $\tilde{q}>2$ follows easily. 

Finally, one can check easily that if $\tilde{q}=\tilde{p}+1$ or $\tilde{p}>\tilde{q}>2$ and $\tilde{p}=-1 \pmod{\tilde{q}}$, then $\tilde{q}^\prime=\tilde{q}-1$ must be true, from which it follows easily that 
$B_{p,q}=B_{\tilde{p},\tilde{q}}-1$ (simply by reversing the derivations above). This finishes the proof.

\end{proof}

Next, we shall upgrade the recursive relation into a form which is more convenient to use, because recursively, $R_{p,q}$ is better behaved (cf. Lemma 2.1). Compare also with Lemma 6.11 in \cite{C}.

\begin{lemma}
Assume $1<p<q$ where $gcd(p,q)=1$, and set $\Delta_{p,q}:=B_{p,q}-R_{p,q}$. 
Then $\Delta_{p,q}\in \{0,1\}$. Moreover, $\Delta_{p,q}=1$
if and only if either $q=p+1$ or $p>2$ and $q=-1\pmod{p}$. 
\end{lemma}

\begin{proof}
Case (1): suppose $q=p+1$. By a direct calculation, one has $B_{p,q}=B_{p,p+1}=p+1$ and 
$R_{p,q}=R_{p,p+1}=p$, which implies that $\Delta_{p,q}=1$.

\vspace{2mm}

Case (2): suppose $q\neq p+1$, and we write $q=kp+r$ where $1\leq r<p$. 

\vspace{2mm}

(i): assume $r=1$, where $k>1$ is necessarily true. In this case, by Lemma 2.2, then Lemma  2.1, one has $B_{p,q}=B_{p, p+1}-1=R_{p,p+1}=R_{p,q}$, which implies that $\Delta_{p,q}=0$ in this case.

\vspace{2mm} 

(ii): assume $1<r=p-1$, where $p>2$ is necessarily true. Note that in this case, $p>2$ and $q=-1\pmod{p}$ is true, so we need to prove $\Delta_{p,q}=1$. To see this, we observe that with $r=p-1$, the condition $r>2$ and $p=-1\pmod{r}$ cannot be satisfied. Hence by Lemma 2.2, then
Lemma 2.1, one has 
$$
B_{p,q}=B_{r,p}=B_{p-1, p}=R_{p-1,p}+1=R_{r,p}+1=R_{p,q}+1,
$$ 
which implies $\Delta_{p,q}=1$. (Here we used the fact that $\Delta_{p-1,p}=1$ from Case (1).)

\vspace{2mm} 

(iii): assume $1<r<p-1$. We need to show $\Delta_{p,q}=0$. 
By induction, we assume the lemma is true for $(r,p)$. With this understood, 
we observe that by Lemma 2.2,
\[
  B_{p,q}=
  \begin{cases}
    B_{r,p}-1, & \text{if } r>2 \text{ and } p= -1\pmod r,\\
    B_{r,p}, & \text{if otherwise, i.e., either } r=2 \text{ or } p\neq -1\pmod r.
  \end{cases}
\]
In the former case, since the lemma is true for $(r,p)$, we have $B_{r,p}=R_{r,p}+1=R_{p,q}+1$,
as $\Delta_{r,p}=1$ by the induction assumption and $R_{r,p}=R_{p,q}$ by Lemma 2.1. Consequently, 
$$
\Delta_{p,q}=B_{p,q}-R_{p,q}=B_{r,p}-1-R_{p,q}=0.
$$
Likewise, in the latter case, we have $B_{r,p}=R_{r,p}=R_{p,q}$ instead, as in this case, since
$p\neq r+1$ is also true, one has $\Delta_{r,p}=0$ by the induction assumption. It follows that 
$$
\Delta_{p,q}=B_{p,q}-R_{p,q}=B_{r,p}-R_{p,q}=0
$$
as well. This finishes the proof. 
\end{proof}

Finally, we proceed further by deriving the following expression for $B_{p,q}$.

\begin{lemma}
Let $1<a<b$ and $gcd(a,b)=1$. Furthermore, we assume $b=ka+s$ where $1<s<a$. Then
$B_{a,b}=\lfloor \frac{a}{\min (u,a-u)}\rfloor$, where $us=1 \pmod{a}$, $1<u<a$. 
\end{lemma}

\begin{proof}
Let $a^\prime,b^\prime$ be the integers satisfying
$$
ab^\prime+ba^\prime=ab+1, 1\leq a^\prime<a, 1\leq b^\prime<b.
$$
Then $B_{a,b}=\max (\lfloor \frac{a}{a-a^\prime}\rfloor , \lfloor \frac{b}{b-b^\prime}\rfloor)$. Next, note that $ba^\prime=1 \pmod{a}$ and $b=s \pmod{a}$, which implies $u=a^\prime$ easily. 
Consequently, $B_{a,b}=\max (\lfloor \frac{a}{a-u}\rfloor , \lfloor \frac{b}{b-b^\prime}\rfloor)$. 

To proceed further, we note that $ab^\prime+bu=ab+1$, which implies 
$$
\frac{b}{b-b^\prime}=\frac{a}{u}+\frac{1}{u(b-b^\prime)}.
$$
As $u>1$ and $b-b^\prime=\frac{1}{a}(bu-1)>\frac{1}{a}((a+1)u-1)>1$, it is easily seen that 
$\lfloor \frac{b}{b-b^\prime}\rfloor =\lfloor \frac{a}{u} \rfloor$,
and $B_{a,b}= \lfloor \frac{a}{\min (u,a-u)}\rfloor$ follows immediately. This finishes the proof. 

\end{proof}

{\bf Proof of Theorem 1.1:}

\vspace{2mm}

For (1), by Lemma 2.1(2), we have $R_{p,q}=\ell(p,q)$. Consequently, 
$$
m_{p,q}=pq+B_{p,q}=pq+ R_{p,q}+\Delta_{p,q}=pq+\ell(p,q)+\Delta_{p,q},
$$
from which part (1) of Theorem 1.1 follows immediately by Lemma 2.3.

For (2), it suffices to prove that $R_{p,q}=\lfloor \frac{q-1}{\min(u,q-u)} \rfloor$. Consider first the case
where $p=q-1$. In this case, we have $R_{p,q}=p$. On the other hand, note that $u=q-1$, so that
$\lfloor \frac{q-1}{\min(u,q-u)} \rfloor=q-1$.  Hence $R_{p,q}=\lfloor \frac{q-1}{\min(u,q-u)} \rfloor$. 

Next, assume $p<q-1$. Let $a=q$, $b=q+p$. Then by Lemma 2.4, 
$B_{a,b}=\lfloor \frac{q}{\min(u,q-u)} \rfloor$. On the other hand, since $1<p<q-1$, 
$B_{a,b}=R_{a,b}$ by Lemma 2.3. It follows easily that
$$
R_{p,q}=R_{a,b}=B_{a,b}=\lfloor \frac{q}{\min(u,q-u)} \rfloor=\lfloor \frac{q-1}{\min(u,q-u)} \rfloor,
$$
where in the last equation, we used the fact that $1<u<q-1$ and that $\frac{q}{\min(u,q-u)}$ is not an integer. This finishes the proof of Theorem 1.1.


\section{Classification of the gap patterns: proof of Theorem 1.5}

Throughout this section, we adopt the following notations: let $C$ denote a rational bicuspidal curve belonging to one of the two families under consideration, and for $z=0$ or $z=\infty$, let $C_z$ be the rational unicuspidal curve obtained by resolving the singularity of $C$ at $z$, and finally, denote by $G_z$ the gap for the curve $C_z$. Our goal is to determine the values of $G_0$ and $G_\infty$. 

\vspace{2mm}

{\it Case (1)}: $C$ belongs to the family in $\C\P^2$, where $q-p>1$. With this understood, consider first the case of $z=0$. We observe that 
$C_0\cdot C_0=q^2-M_{p,q}=q(q-p)+ R_{p,q}$, which implies that
$$
G_0=m_{q-p,q}-C_0\cdot C_0=(q-p)q+B_{q-p,q}-q(q-p)-R_{p,q}=B_{q-p,q}-R_{q-p,q}=\Delta_{q-p,q}.
$$
Here we used the identity $R_{p,q}=R_{q-p,q}$ from Lemma 2.1(1). By a similar argument, $G_\infty=\Delta_{p,q}$. With this understood, it is easy to see that 
Theorem 1.5(1) follows immediately from Lemma 2.3, after we verify that case (i) and case (ii) are mutually exclusive.
To see this, note that $p$ and $q-p$ are co-prime, so that if $q+1$ is divisible by both $p$ and $q-p$,
it must be divisible by $p(q-p)$ as well. On the other hand, since $p>2$ and $q-p>2$, one has 
$p(q-p)> p+q-p+1=q+1$, which contradicts $q+1$ being divisible by $p(q-p)$. This finishes Case (1).

\vspace{2mm}

{\it Case (2)}: $C$ belongs to the family in $\C\P^1\times \C\P^1$. To ease the notation, we set
$Q:=2\alpha p-q$. Then note that $C\cdot C=2\alpha p^2=Qp+ qp$, which implies by a similar argument as in Case (1) the following identities:
$$
G_0=B_{p,Q}-R_{p,q}, \;\;\; G_\infty =B_{p,q}-R_{p,Q}. 
$$

With the preceding understood, we assume $p>2$ first. We divide the discussions according to the following mutually exclusive scenarios.

(i) $q-p=1$. In this case, $Q=2\alpha p-p-1$. It follows from Lemmas 2.3 and 2.1 that 
$$
B_{p,Q}=R_{p,Q}+1=R_{p,p-1}+1=p-1+1=p.
$$
On the other hand, $R_{p,q}=R_{p,p+1}=p$, which implies $G_0=0$. For the calculation of $G_\infty$, we note that by Lemma 2.3, $B_{p,q}=R_{p,q}+1=R_{p,p+1}+1=p+1$. On the other hand, 
$R_{p,Q}=p-1$, which implies $G_\infty=2$.

\vspace{2mm}

(ii) $q-p>1$ and $p | (q-1)$. By the same argument, $B_{p,Q}=p$, $R_{p,q}=p$, so that $G_0=0$.
On the other hand, by Lemmas 2.3 and 2.1 we have $B_{p,q}=R_{p,q}=R_{p,p+1}=p$ instead. With
$R_{p,Q}=p-1$ continuing to be true, we obtain $G_\infty =1$ in this case.

\vspace{2mm}

(iii) $p | (q+1)$. In this case, $Q=2\alpha p-kp+1$ for some $k\leq \alpha$. Similarly, using 
Lemmas 2.3 and 2.1, one has $B_{p,Q}=R_{p,Q}=R_{p,p+1}=p$. With $R_{p,q}=R_{p,p-1}=p-1$, it follows that $G_0=1$. For $G_\infty$, we note that by Lemmas 2.3 and 2.1, $B_{p,q}=R_{p,q}+1=R_{p,p-1}+1=p-1+1=p$.
With $R_{p,Q}=R_{p,p+1}=p$, we obtain $G_\infty=0$.

\vspace{2mm}

(iv) None of the above. In this case, it is easy to see that there are $r,s$ such that
$$
Q=r, q=s \pmod{p}, \;\; 1<r<p-1, 1<s<p-1.
$$
We observe that $r+s=p$. With this understood, one has by Lemmas 2.3 and 2.1,
$$
B_{p,Q}=R_{p,Q}=R_{r,p}, \; B_{p,q}=R_{p,q}=R_{s,p}. 
$$
With $R_{r,p}=R_{s,p}$ as $s=p-r$ (cf. Lemma 2.1(1)), it follows easily that $G_0=G_\infty=0$. 

It remains to consider the case where $p=2$. If $q=3$, one has 
$$
B_{p,Q}=R_{p,Q}=R_{2,4\alpha-3}=2 \mbox{ and } R_{p,q}=R_{2,3}=2,
$$
which implies $G_0=0$. On the other hand, $B_{p,q}=R_{p,q}+1=R_{2,3}+1=3$, and $R_{p,Q}=R_{2,4\alpha-3}=2$, which implies $G_\infty=1$. For the case of $q>3$, one continues to have $B_{p,Q}=2$ and  
$R_{p,q}=2$, which gives $G_0=0$, but for $G_\infty$, one has $B_{p,q}=R_{p,q}=2$ and 
$R_{p,Q}=2$, so that $G_\infty=0$. 

The proof of Theorem 1.5 is completed. 


\vspace{2mm}

{\Small Weimin Chen, University of Massachusetts, Amherst.\\
{\it E-mail:} wch@umass.edu 

\vspace{2mm}

{\Small Yusupjan Ouyang, University of Massachusetts, Amherst.\\
{\it E-mail:} yuzeouyang@umass.edu 

\vspace{2mm}

{\Small Sam Silver, University of Massachusetts, Amherst.\\
{\it E-mail:} samuelsilver@umass.edu


\begin{thebibliography}{}
\bibitem{C} W. Chen, {\em Small symplectic $4$-manifolds via contact gluing and some applications}, 
arXiv:2503.05932v6 [math.GT] 5 June 2026.
\bibitem{F} J. Fernández de Bobadilla, I. Luengo, A. Melle-Hernández, and A. Némethi, {\em Classification of rational unicuspidal projective curves whose singularities have one Puiseux pair}, in Proceedings of Sao Carlos Workshop 2004 Real and Complex Singularities, Series Trends in Mathematics (Birkhäuser, 2007), 31-46.
\bibitem{FZ} H. Flenner and M. Zaidenberg, {\em On a class of rational cuspidal plane curves},
Manuscripta Math. {\bf 89}(1996) no. 4, 439-459.
\bibitem{GS} M. Golla and L. Starkston, {\em Rational cuspital curves and symplectic fillings}, arXiv:2111.09700v1 [math.GT] 18 Nov 2021.
\bibitem{Hartsh} R. Hartshorne, {\em Curves with high self-intersection on algebraic surfaces}, 
Publ. Math. I.H.E.S., {\bf 36} (1969), 111-125.
\bibitem{LM} T.-J. Li and C.-Y. Mak, {\em Symplectic divisorial capping in dimension $4$}, 
Journal of Symplectic Geometry {\bf 17} (2019), no.6, 1835-1852.
\bibitem{Moe} T. K. Moe, {\em Cuspidal curves on Hirzebruch surfaces}, Ph.D. thesis, University of Oslo, 2013. 
\end{thebibliography}
\end{document}